\documentclass[11pt,a4paper]{amsart}
\usepackage{mathptmx} 
\usepackage[scaled=0.90]{helvet} 
\usepackage{courier}
\normalfont
\usepackage[T1]{fontenc}

\usepackage{amssymb}
\usepackage{tikz-cd}

\renewenvironment{proof}{{\bf \emph{Proof.} }}{\hfill  \\} 
\renewcommand{\epsilon}{\varepsilon}
\renewcommand{\setminus}{\smallsetminus}
\renewcommand{\emptyset}{\varnothing}

\newtheorem{theorem}{Theorem}[section]
\newtheorem{proposition}[theorem]{Proposition}
\newtheorem{corollary}[theorem]{Corollary}
\newtheorem{lemma}[theorem]{Lemma}

\theoremstyle{definition}

\theoremstyle{remark}
\newtheorem{remark}[theorem]{Remark}

\newcommand{\U}{\mathcal U}

\newcommand{\FP}{\operatorname{FP}}

\newcommand{\Ho}{\operatorname{H}}
\newcommand{\cohom}[3]{H^{{\raise1pt\hbox{$\scriptstyle#1$}}}(#2\>\!,#3)}
\newcommand{\tatecohom}[3]%
  {\widehat H^{{\raise1pt\hbox{$\scriptstyle#1$}}}(#2\>\!,#3)}

\newcommand{\Cohom}[3]%
  {H^{{\raise1pt\hbox{$\scriptstyle#1$}}}\big(#2\>\!,#3\big)}
\newcommand{\Tatecohom}[3]%
  {\widehat H^{{\raise1pt\hbox{$\scriptstyle#1$}}}\big(#2\>\!,#3\big)}

\newcommand{\homol}[3]{H_{{\lower1pt\hbox{$\scriptstyle#1$}}}(#2\>\!,#3)}
\newcommand{\homolog}[2]{H_{{\lower1pt\hbox{$\scriptstyle#1$}}}(#2)}

\renewcommand{\ker}{\operatorname{Ker}}

\newcommand{\mono}{\rightarrowtail}
\newcommand{\epi}{\twoheadrightarrow}

\title{Subdirect sums of Lie algebras}
\author{D.~ H. ~Kochloukova}

\address{Dessislava H.~Kochloukova, Department of Mathematics, University of Campinas,
13083-859, Campinas, SP, Brazil}\email{desi@ime.unicamp.br}

\author{C.~Mart\'inez-P\'erez}

\address{Conchita Mart\'inez-P\'erez, Departamento de Matem\'aticas, Universidad de Zaragoza,
50009 Zaragoza, Spain} \email{conmar@unizar.es}

\thanks{The first names author was supported by  ``bolsa de produtividade em pesquisa'' CNPq, 303350/2013-0, CNPq, Brazil and by FAPESP, 2016/05678-3, Brazil. The second named author was supported by Gobierno de Aragon, European Regional Development Funds and
MTM2015-67781-P (MINECO/FEDER)}

\date{\today} 

\thanks{}

\begin{document}

\thispagestyle{empty}

\begin{abstract} We show Lie algebra versions of some results on homological finiteness properties of subdirect products of groups. These results include a version of the 1-2-3 Theorem.
\end{abstract}

\maketitle

\centerline{\sl Dedicated to the memory of Gilbert Baumslag}

\section{Introduction}

The aim of this paper is to explore subdirect sums of Lie algebras from the point of view of their homological properties, in the same spirit of many recent results for subdirect products of groups. We shall look at two types of results: On the one hand at certain theorems on subdirect products of  free groups \cite{BaumRose} that were recently generalised to the case of subdirect products of limit groups \cite{BHMS-1}, \cite{BHMS}; on the other at fibre sums of Lie algebras and Lie versions of the $n$-$n+1$-$n+2$ Conjecture and the homotopical $n$-$n+1$-$n+2$ Conjecture (\cite{Beno}, \cite{Francismar}).

For both types of results, the theory of subdirect products of groups was developed using a lot of homotopical methods that are not available in the category of Lie algebras and furthermore  some of the auxiliary statements of the homotopical approach do not make sense for Lie algebras since the group version uses the idea of subgroups of finite index. One simple example in this direction is the fact that any non-trivial element of a free group of finite rank is a part of a free basis of a subgroup of finite index. In the case of Lie algebras there is no analogous of the finite index notion. 

In this paper we adopt homological methods to study subdirect sums of Lie algebras of homological type $FP_s$. All the Lie algebras in this paper are over a field $K$
of arbitrary characteristic.
We define the subdirect Lie sum 
$$
L \leq L_1 \oplus \ldots \oplus L_k
$$
of Lie algebras $L_1, \ldots L_k$  as a Lie subalgebra $L$ of the direct sum $L_1 \oplus \ldots \oplus L_k$ such that $L$ projects surjectively on $L_i$ for every $1 \leq i \leq k$.  A Lie algebra $L$ is of type $FP_s$ if the trivial $\U(L)$-module $K$ has a projective resolution with all projective modules finitely generated in dimension $ \leq s$ i.e.
there is an exact complex
$${\mathcal P} : \ldots \to P_i \to P_{i-1} \to \ldots \to P_1 \to P_0 \to K \to 0
$$
such that $P_i$ is a finitely generated projective $\U(L)$-module for $i \leq s$, where $\U(L)$ denotes the universal enveloping algebra of $L$.
 Note that if $L$ is a Lie algebra of homological type $FP_s$, then  $\dim \Ho_i(L, K) < \infty$  for $i \leq s$.
Little is known about the homological type $FP_s$ for Lie algebras.  A complete classification of  metabelian Lie algebras of type $FP_2$ was discovered by Bryant and Groves in \cite{BryantGroves1}, \cite{BryantGroves2} in terms of the so called Bryant-Groves invariant  $\Delta$. It turned out that for a finitely generated metabelian Lie algebra the homological property  $FP_2$ and finite presentability are equivalent. A classification of split metabelian Lie algebras $L$ of type $FP_s$ also in terms of $\Delta$  was  proved by Kochloukova in \cite{Desi}. By definition $L$ is a {\bf split} metabelian Lie algebra if there is a short exact split sequence of Lie algebras $A \rightarrowtail L \epi Q$ with $A$ and $Q$ abelian. In \cite{GrovesKoch} Groves and Kochloukova proved that the finitely presented soluble Lie algebras of type $FP_{\infty}$ are  finite dimensional.

Our first result is on subdirect sums $L \leq L_1 \oplus \ldots \oplus L_k$  when each $L_i$ is a non-abelian free Lie algebra.

\medskip
{\bf Theorem A.} {\it Let $s \geq 2$ be a natural number and $$L \leq F_1 \oplus \ldots \oplus F_k$$ be a subdirect sum  of finitely generated non-abelian free Lie algebras $F_1, \ldots, F_k$ such that $L$ is of type $FP_s$ and $L\cap F_i\neq 0$ for $1\leq i\leq k$. Then for every $ 1 \leq i_1 <  i_2 < \ldots < i_s \leq k$ we have that
$p_{i_1, \ldots, i_s}(L) = F_{i_1} \oplus \ldots \oplus F_{i_s}$.}

\medskip

A group theoretic version of  Theorem A is a special case of a result of Kochloukova on subdirect product of non-abelian limit groups \cite{Desi2} and it states that the projection of the subdirect product  on the direct product of $s$ coordinates groups  is {\bf virtually }  surjective. Note that in the Lie algebra version the word virtually disappears due to the phenomena mentioned before.
\medskip

As a corollary we obtain the following result.

\medskip
{\bf Theorem B.}
{\it Let $F_1,\ldots,F_k$ be finitely generated  free Lie algebras and assume that 
$$L\leq F_1\oplus\ldots\oplus F_k$$
is a subdirect Lie sum with $L$ of type $FP_k$. Then $L$ is a direct sum of at most $k$ free Lie algebras.}

\medskip
 A group theoretic version of Theorem B was established by Baumslag and Roseblade in \cite{BaumRose}.  This was generalised to the case of subdirect products of non-abelian surface groups by Bridson, Howie, Miller and Short in \cite{BHMS-1} and later  the same authors resolved  the case of subdirect product of limit groups in \cite{BHMS}. All three papers mix homological with homotopical methods but in all cases there are homotopical ingredients that do not pass to the category of Lie algebras.

In the rest of the paper we consider  fibre sums of Lie algebras. Let 
$A \rightarrowtail L_1 \buildrel{\pi_1}\over\epi Q$ and $B \rightarrowtail L_2 \buildrel{\pi_2}\over\epi Q$ be short exact sequences of Lie algebras. The fibre sum $P$ of these two short exact sequences is the Lie algebra
$$
P = \{ (h_1, h_2) \in L_1  \oplus L_2 \mid \pi_1(h_1) = \pi_2(h_2) \}
$$
By definition $P$ is a Lie subalgebra of the direct sum $L_1 \oplus L_2$.

There are homotopical and homological Conjectures on fibre products of groups (see \cite{Beno}, \cite{Francismar}) that suggest when the fibre has homotopical type $F_s$ and  homological type $FP_s$. If $s \geq 2$, a group has a homotopical type $F_s$ if it is finitely presented and has the homological type $FP_s$.
Motivated by the existing homotopical and homological Conjectures on fibre products of groups  we suggest the following two conjectures:

\medskip
{\bf The  $n$-$n+1$-$n+2$ Conjecture for Lie algebras.}  {\it Let $n \geq 1$ be a natural number. Let $$A \rightarrowtail L_1  \buildrel{\pi_1}\over\epi Q  \hbox{ and }
B \rightarrowtail L_2  \buildrel{\pi_2}\over\epi Q$$ be short exact sequences of Lie algebras with  both $L_1$ and $L_2$ of homological type $FP_{n+1}$ and finitely presented, $A$ of type $FP_n$  and $Q$ of type $FP_{n+2}$. Then the fibre Lie sum
$$
P = \{ (h_1, h_2) \in L_1  \oplus L_2 \mid \pi_1(h_1) = \pi_2(h_2) \}
$$
has type $FP_{n+1}$ and is finitely presented.}

\medskip
{\bf Remark.} {\it Since $L_1$ is a finitely presented Lie algebra with an ideal $A$ that is finitely generated as a Lie algebra we have that $Q \simeq L_1 / A$ is a finitely presented Lie algebra.}

\medskip
Note that whereas the statement of the original conjecture in the group case when $n \geq 2$ assumes that  $A$ is finitely presented, we are not assuming this. 
 In the case of groups more information on  the $n$-$n+1$-$n+2$ Conjecture can be found in \cite{Beno}. The $n$-$n+1$-$n+2$ Conjecture for groups holds for $n = 1$, in this case it is called the 1-2-3 Theorem and an algorithmic proof was presented by Bridson, Howie, Miller and Short in \cite{BHMS2}.  The symmetric case of the 1-2-3 Theorem  was  established earlier by Baumslag, Bridson, Miller and Short in \cite{BBMS}. As pointed by Kuckuck in \cite{Beno} the $n$-$n+1$-$n+2$  Conjecture for groups implies another conjecture  i.e. the Virtual Surjection Conjecture. 
It is our believe that both the ordinary and the homological version of 
the $n$-$n+1$-$n+2$ Conjecture for groups hold. More on the homological version of the $n$-$n+1$-$n+2$ Conjecture for groups could be found in \cite{Francismar}. 
We state below the homological version of the Conjecture for Lie algebras.

\medskip
{\bf The homological  $n$-$n+1$-$n+2$ Conjecture for Lie algebras.}  {\it Let $n \geq 1$ be a natural number. Let  $$A \rightarrowtail L_1  \buildrel{\pi_1}\over\epi Q  \hbox{ and }
B \rightarrowtail L_2  \buildrel{\pi_2}\over\epi Q$$
 be short exact sequences of Lie algebras with  both $L_1$ and $L_2$ of homological type $FP_{n+1}$, $A$ of type $FP_n$  and $Q$ of type $FP_{n+2}$. Then the fibre Lie sum
$$
P = \{ (h_1, h_2) \in L_1  \oplus L_2 \mid \pi_1(h_1) = \pi_2(h_2) \}
$$
has type $FP_{n+1}$.}

\medskip

In this paper we show that the first conjecture for Lie algebras holds for $n = 1$. It is worth noting that the proof in the group case is homotopical and uses significantly the second homotopy group of a CW-complex plus some results on Pheifer movements \cite{BBMS}. In the Lie algebra case these techniques are not available.
 
\medskip
{\bf Theorem C.} {\it The $1$-$2$-$3$ Conjecture for Lie algebras holds. Furthermore  the $1$-$2$-$3$-Conjecture holds without assuming that $Q$ is of type $FP_3$ if $H_2(A,K)$ is finitely generated as $\U(Q)$-module.}

\medskip
We obtain the following corollaries of Theorem C.

\medskip
{\bf Corollary D1.} {\it Let $$L \leq L_1 \oplus \ldots \oplus L_k$$  be a subdirect Lie sum  with $L \cap L_i \not= 0$ and $L_i$ finitely presented Lie algebra for all $i$. Assume that $p_{i,j} (L) = L_i \oplus L_j$ for all $1 \leq i < j \leq k$, where  $p_{i,j} : L_1 \oplus \ldots \oplus L_k \to L_i \oplus L_j$ is the canonical projection. Then $L$ is a finitely presented Lie algebra. }

\medskip
{\bf Corollary D2.} {\it Let $$L \leq L_1 \oplus \ldots \oplus L_k$$ be a subdirect Lie sum   with $L \cap L_i \not= 0$ and $L_i$ free non-abelian Lie algebra  for any $i$. Then $L$ is a finitely presented Lie algebra if and only if   $p_{i,j} (S) = L_i \oplus L_j$ for all $1 \leq i < j \leq k$.}

\medskip
We also consider the homological $1$-$2$-$3$-Conjecture for Lie algebras and show

\medskip
{\bf Theorem E.} {\it The homological $1$-$2$-$3$-Conjecture for Lie algebras holds if in addition we assume that $Q$ is a finitely presented Lie algebra.}

\medskip
From this we get the following two corollaries.

\medskip
{\bf Corollary F1.}  {\it Let $$L \leq L_1 \oplus \ldots \oplus L_k$$ be a subdirect Lie sum with $L \cap L_i \not= 0$ and $L_i$ of type $FP_2$  for all $i$. Assume that $p_{i,j} (L) = L_i \oplus L_j$ for all $1 \leq i < j \leq k$, where  $p_{i,j} : L_1 \oplus \ldots \oplus L_k \to L_i \oplus L_j$ is the canonical projection. Then $L$ is a Lie algebra of type $FP_2$.}

 \medskip 
{\bf Corollary F2. }
{\it Let  $$L \leq L_1 \oplus \ldots \oplus L_k$$ be a subdirect Lie sum 
 with $L \cap L_i \not= 0$ and $L_i$ free non-abelian Lie algebra for all $i$. Then $L$ is finitely presented as a Lie algebra if and only if $L$ is of type $FP_2$.}

\medskip
As shown by Bridson, Howie, Miller and Short Corollary D1 and Corollary D2 hold in the case of groups with surjection substituted by virtual surjection, see \cite{BHMS}. The group theoretic versions of Theorem E, Corollary F1 and Corollary F2, where in the corollaries surjection is substituted with virtual surjection, were recently proved by Kochloukova and Lima in \cite{Francismar}. We finish with a Lie algebra version of another recent result from \cite{Francismar}

\medskip
{\bf Theorem G.} {\it 
a) The homological $n$-$n+1$-$n+2$-Conjecture for Lie algebras holds if the second short exact  sequence $B \rightarrowtail L_2 \epi Q$ splits.

b) If the homological $n$-$n+1$-$n+2$-Conjecture for Lie algebras holds whenever $L_2$ is a free Lie algebra then it holds for any Lie algebra $L_2$.}

\medskip
 Note that a group theoretic homotopical version of Theorem G was proved by Kuckuck in \cite{Beno}. The methods used in \cite{Beno} are homotopical and use the Borel construction and the theory of stacks, so are not aplicable for Lie algebras. An alternative approach (using spectral sequence arguments) to a homological group theoretic version was developed in \cite{Francismar}, here we adapt this approach to Lie algebras.

We state a Lie algebra version of the Virtual Surjection Conjecture from \cite{Beno} 

\medskip
{\bf The Homological Surjection Conjecture for Lie algebras.} {\it Let  $m,k$ be natural numbers such that $ 2 \leq m \leq k$ and $$L \leq L_1 \oplus \ldots \oplus L_k$$ be a subdirect Lie sum 
 with $L \cap L_i \not= 0$, where each Lie algebra $L_i$ is $FP_m$ and for every $1 \leq i_1 < i_2 \ldots < i_m \leq m$ we have for the canonical projection $p_{i_1, \ldots, i_m} : L_1 \oplus \ldots \oplus L_k \to L_{i_1} \oplus \ldots \oplus L_{i_m}$ that $p_{i_1, \ldots, i_m}(L) = L_{i_1} \oplus \ldots \oplus L_{i_m}$. Then $L$ is a Lie algebra  of type $FP_m$.}

\medskip
{\bf Theorem H.} {\it If the homological $m-1$-$m$-$m+1$-Conjecture for Lie algebras holds then the Homological Surjection Conjecture for Lie algebra holds too.} 

\section{ Preliminaries on Lie algebras : finite presentability and homological properties of Lie algebras}

Recall that every Lie algebra $L$ over a field $K$  embeds in its universal enveloping algebra $\U(L)$. All the Lie algebras in this paper are over a field $K$
of arbitrary characteristic.
The universal enveloping algebra $\U(L)$ is an associative $K$-algebra, where for $a,b \in L$ we have $[a,b] = ab - ba$. If $L$ is a free Lie algebra with a free basis $X$ then $\U(L)$ is a free associative $K$-algebra with a basis $X$. 
The Poincare-Birkhoff-Witt theorem produces a basis of $\U(L)$ as a vector space over $K$ \cite{Bakhturin}, \cite{Jacobson}. For general homological results for Lie algebras, the reader is referred to  \cite[Section~7]{Weibel}.

\subsection{Finitely presented Lie algebras}

In the category of Lie algebras for any set $X$ of arbitrary cardinality there exists a free object $F(X)$  with $X$ as free basis \cite{Bakhturin}. A free Lie algebra $F = F(X)$ has a universal enveloping algebra $\U(F)$ which is a free associative algebra  with basis $X$. Every Lie algebra $L$  is isomorphic to $F / R$, where $F$ is a free Lie algebra and $R$ is an ideal of $F$. If  there is  such an isomorphism and $F=F(X)$ with $X$ finite then we say that $L$ is finitely generated and if moreover $R$ can be generated as an ideal by finitely many elements then we say that $L$ is finitely presented.

\begin{lemma}\label{propfree} Let $F$ be a free Lie algebra. Then
\begin{itemize}
\item[i)]( \cite{Sirsov}, \cite{Witt} ) any Lie subalgebra of $F$ is free.

\item[ii)] (\cite[theorem~3]{BBaumslag}) if $F$ is finitely generated and $0\neq I$ is an ideal finitely generated as subalgebra then $I=F$.
\end{itemize}
\end{lemma}

Note that if we consider Lie rings  (i.e. $K$ is only a commutative ring and not a field) then a Lie subring of a free Lie ring is  not necessary free (see \cite[ex.~2.4.1]{Bakhturin}).

\subsection {Homology of Lie algebras} 

For a Lie algebra $L$ over a field $K$ the $n$-th homology of $L$ with coefficients in a $\U(Q)$-module $V$ is
$$
H_n(L, V) : = Tor_n^{\U(L)} (K, V),
$$
where we are denoting by $K$ not only the coefficient field but also the trivial $\U(L)$-module i.e. the module on which $L$ acts as 0.
It is easy to see that
$$
H_1(L, K) \simeq L/[L,L].
$$

The following result is a Lie algebra version of the Hopf formula, it is a particular case of a result from  \cite{B-C}.  Alternatively an algebraic argument  using Gruenberg 
resolution can be used  to prove the Hopf formula  for second homology of a group, see \cite{Rotman}. The same argument works in the Lie case.

\begin{lemma} \label{schur} (Hopf type formula)  Let $F$ be a free Lie algebra and $R$ an ideal of $F$. Then
$$
H_2(F/R,K) \simeq R \cap [F,F]/ [R, F].
$$
\end{lemma}

By \cite[theorem.~11.31]{Rotman} we have a K\"unneth type formula
$$
H_n (L_1 \oplus L_2, K) \simeq \oplus_{0 \leq i \leq n} H_i(L_1, K) \otimes_K H_{n-i}(L_2, K)
$$
(the $Tor_1$ part is missing since tensoring over the field $K$ is an exact functor).

\subsection{Homological properties of Lie algebras}
Assume that $L$ is a Lie algebra of type $FP_s$ and consider  a projective resolution of the trivial $\U(L)$-module $K$ $${\mathcal P} : \ldots \to P_i \to P_{i-1} \to \ldots \to P_1 \to P_0 \to K \to 0
$$
such that $P_i$ is a finitely generated projective $\U(L)$-module for $i \leq s$. As in the group case if there is a projective resolution with this property then there is a free resolution with the same property i.e. we can assume that each $P_i$ is a finitely generated free $\U(L)$-module for $i \leq s$.

As we said in the introduction, little is known about the homological property $FP_s$ for Lie algebras. In some cases it is easy to get Lie algebra versions of well known results for groups as in the following versions of \cite[proposition~2.1,proposition~2.2]{Bieribook}.

\begin{lemma} \label{FP1} A Lie algebra $L$ is of type $FP_1$ if and only if $L$ is finitely generated as a Lie algebra.
\end{lemma}

Let $F$ be a free Lie algebra and $R$ an ideal of $F$. We define the adjoint action of $f \in F$ on $R$ by $r \circ f = [r,f]$. With this action, $R$ is a (right) $\U(F)$-module and this induces also an action of $\U(L)$ on $R/ [R,R]$ for $L=F/R$. We have, as in the group case.

\begin{lemma}\label{FP2} Let $L=F/R$ with $F$  a finitely generated free Lie algebra  and $R$ an ideal of $F$. Then $L$ is of type $FP_2$ if and only if  $R/ [R, R]$ is finitely generated as $\U(L)$-module via the adjoint $L$-action. \end{lemma}   

We omit the proofs of Lemma \ref{FP1} and Lemma \ref{FP2} as the proofs of \cite[proposition~2.1]{Bieribook} and \cite[proposition~2.2]{Bieribook} are homological and the same proofs work substituting   group algebras with universal enveloping algebras.
Note that \cite[theorem~1.3]{Bieribook}  applied to the associative ring $\U(L)$ and the trivial $\U(L)$-module $K$ gives immediately:

\begin{theorem}  \label{LieBieri} Let $s \geq 1$ be a natural number.  Then a Lie algebra $L$ is of type $FP_s$ if and only if  $$Tor_k^{\U(L)}(\prod \U(L), K) \simeq  \prod Tor_k^{\U(L)}( \U(L), K) $$ for all $k \leq s-1$ and for every direct product $\prod \U(L)$. In particular,  if $L$ is finitely generated then $L$ is of type $FP_s$ if and only if  $$Tor_k^{\U(L)}(\prod \U(L), K) = 0$$ for all $1 \leq k \leq s-1$ and for every direct product $\prod \U(L)$.
\end{theorem}

Using Theorem \ref{LieBieri} and modifying the proof of \cite[proposition~2.7]{Bieribook} by substituting group algebras with enveloping algebras of Lie algebras we obtain the following result.

\begin{theorem} \label{FPinfty} Let $m \geq 1$ be a natural number and let 
$A \to B \to C$ be a short exact sequence of Lie algebras such that $A$ is of type $FP_{\infty}$. Then $B$ is of type $FP_m$ if and only if $C$ is of type $FP_m$.
\end{theorem}

\section{Modules of free Lie algebras with finite dimensional first homology}

In this section we prove a technical result that we will use later.

\begin{lemma}\label{vanishing} Let $F$ be a free Lie algebra with universal algebra $\U(F)$ and $V$ a left  $\U(F)$-module. Assume that there is some $f\in\U(F)\setminus\{0\}$ such that  $\dim fV<\infty$ and that $\dim\Ho_1(F,V)<\infty$. Then $\dim V<\infty$.
\end{lemma}
\begin{proof} To compute the homology of $F$ we may use the standard free resolution of right $\U(F)$-modules
$$0\to\oplus\U(F)\to\U(F)\to K \to 0$$
where the cardinality of the direct sum is the rank of $F$. 
Let $\{ a_1,\ldots,a_j,\ldots \}_{j \in J}$ be a basis of $F$ as a free Lie algebra. We denote 
$$\oplus\U(F)=\oplus_{j} e_j\U(F)$$ and then the map $\oplus\U(F)\to\U(F)$ is given by $e_j\mapsto a_j$. The homology of $F$ with coefficients in $V$ is the homology of the sequence obtained after applying the functor $- \otimes_{\U(F)}V$ to the deleted free resolution, i.e., 
$$0\to\oplus_{j \in J} e_jV\buildrel \varphi\over\to V \to 0$$
thus 
\begin{equation} \label{ker-dim} \dim\Ho_1(F,V)=\dim\ker\varphi<\infty.
\end{equation}

We write $f$ associatively, i.e., opening all possible brackets via $[a,b]=ab-ba$. Then for some natural number $q$ and some $j_1,\ldots, j_q \in J$
 we get
$$f=l+a_{j_1}f_{j_1}+\ldots+a_{j_q}f_{j_q}$$
with $l\in K$ and $f_{j_i} \in\U(F) \setminus \{ 0\}$, otherwise $f = l$ and  there would be nothing to prove. We say that $f$ has no constant term if $l = 0$.

\medskip
{\bf Claim.} {\it  If $f$ has no constant term then
 $\dim f_{j_i} V<\infty$ for any $1\leq i\leq q$.}
 
  \medskip
  {\bf Proof of the Claim:}  
Any element of the form
$$e_{j_1}f_{j_1}v+\ldots+e_{j_q}f_{j_q}v$$
with $v\in V$ is mapped to 
$fv\in fV$
under $\varphi$, thus $\varphi$ induces a linear map
$$\{e_{j_1}f_{j_1}v+\ldots+e_{j_q}f_{j_q}v:v\in V\}\to fV.$$
Observe that the target has finite dimension and the kernel lies inside $\ker\varphi$ thus by (\ref{ker-dim}) it has finite dimension too. We deduce that
$$\dim\{e_{j_1}f_{j_1}v+\ldots+e_{j_q}f_{j_q}v:v\in V\}<\infty.$$
Observe that this set maps epimorphically onto $f_{j_i}V$. Therefore $\dim f_{j_i}V<\infty$ for any $1 \leq i \leq q$. This completes the proof of the claim.

\medskip
To prove the result we come back to the general case when
$$f=l+a_{j_1}f_{j_1}+\ldots+a_{j_q}f_{j_q}$$
and argue by induction on the length $d$ of the longest monomial in $f$ as an element of the free associative algebra $\U(F) = k[a_1, a_2, \ldots, a_j, \ldots]$.  Observe that the length $d$ is strictly  bigger than the length of the longest monomial  of each $f_i$. Pick some $a_{j_i}$, $1\leq i\leq q$ and note that
\begin{equation} \label{inclusion1} [a_{j_i},f]V=\{a_{j_i}fv-fa_{j_i}v : v\in V\}\subseteq a_{j_i}fV+fV.\end{equation}
Since $\dim f V < \infty$, (\ref{inclusion1}) implies that 
\begin{equation} \label{dim01} \dim [a_{j_i},f]V \leq 2 \dim f V <\infty.\end{equation}
Note that
$$[a_{j_i},f]=
a_{j_i}( \sum_{s=1}^q
 a_{j_s}  f_{j_s} )- \sum_{s=1}^q  a_{j_s}f_{j_s}a_{j_i} = 
a_{j_i}( - f_{j_i} a_{j_i} + \sum_{s=1}^q
 a_{j_s}  f_{j_s} )- \sum_{s=1\atop s\neq i}^q  a_{j_s}f_{j_s}a_{j_i}$$
As $[a_{j_i},f]$ has no constant  term, using the Claim we see that
$$\dim f_{j_s}a_{j_i}V<\infty \hbox{ for each } 1\leq s\leq q, s\neq i.$$ 
Note that as the universal enveloping algebra of a Lie algebra does not have  zero-divisors (in our case even $\U(F)$ is a free associative algebra)
we deduce that $f_{j_s} a_{j_i} \not= 0 $ in $\U(F)$. Since $f_{j_s} a_{j_i}$ has no constant term  we may use the Claim again and deduce that there is some $\lambda\in \U(F) \setminus \{  0 \}$ with $\dim \lambda V<\infty$ and such that the length of the longest monomial in $\lambda$ is strictly shorter than $d=$  the length of the longest monomial  in $f$. So by induction we have $\dim V<\infty.$
\end{proof}

\section{Finitely presented subdirect sums of free Lie algebras}

This section contains the proofs of Theorems A and B. We begin with a version of Theorem A for the case when $s=2$ (but observe that the hypothesis here is slightly weaker). Recall that for a Lie algebra $L$, as for groups, we set $\gamma_1(L)=L$ and $\gamma_j(L)=[L,\gamma_{j-1}(L)]$.

\begin{theorem} \label{nilpotent} Let $L \leq F_1 \oplus \ldots \oplus F_k$ be a subdirect sum of finitely generated free non-abelian Lie algebras $F_1, \ldots, F_k$ such that $L\cap F_i\neq 0$ for $1\leq i\leq k$, $L$ is finitely generated and $\dim H_2(L, K) < \infty$. Then

a) for every $1 \leq i < j \leq k$ we have that $p_{i,j}(L) = F_i \oplus F_j$;

b) for every $ 1 \leq i \leq k$ we have that $\gamma_{k-1}(F_i) \subseteq L$.
\end{theorem}

\begin{proof} a) We will show that $p_{k-1, k}(L) = F_{k - 1} \oplus F_k$. Let $\Delta_k$ be the kernel of the projection $L \to F_k$.

It is easy to see that $p_{k-1}(\Delta_k) $ is an ideal of $F_{k-1}$. Indeed if $f \in p_{k-1}(\Delta_k) $ we choose $a \in \Delta_k$ such that $p_{k-1}(a) = f$. Let $\alpha \in F_{k-1}$ then there exists $b \in L$ be such that $p_{k-1}(b) = \alpha$.  Then $[f, \alpha] = p_{k-1}([a,b]) \in p_{k-1}(\Delta_k)$ as claimed. 
Note that $p_{k-1}(\Delta_k) $ is a non-trivial ideal of $F_{k-1}$ as it contains $L \cap F_{k-1}$.
 We are going to show  by a homological argument that the abelianization of $\Delta_k$ is finite dimensional, hence the abelianization of $p_{k-1}(\Delta_k)$ is finite dimensional. Since a free Lie algebra is finitely generated if and only if its abelianization is finite dimensional we deduce that $p_{k-1}(\Delta_k)$ is finitely generated. Then by Lemma \ref{propfree} (ii) $p_{k-1}(\Delta_k) = F_{k-1}$.

To show that the abelianization of $\Delta_k$  is finite dimensional consider  the Lyndon-Hochschild-Serre spectral sequence for Lie algebras \cite[Section~7.5]{Weibel} for the short exact sequence of Lie algebras
$$
\Delta_k \rightarrowtail L \epi F_k. 
$$
Then $E^2_{p,q} = H_p (F_k, H_q( \Delta_k, K))$. Since $F_k$ is a free Lie algebra we have that $H_p(F_k,  - ) $$ = 0$ for $p \geq 2$. Hence $E^2_{p,q} = 0$ for $p \geq 2$ and the spectral sequence collapses and $E^2_{p,q} = E^{\infty}_{p,q}$ and $H_{s}(L, K)$ has a filtration with quotients $E^{\infty}_{p,q}$ for $p + q = s$.
Since $H_2(L, K)$ is finite dimensional we deduce that $E^2_{p,q}$ is finite dimensional  for $p + q = 2$, in particular
\begin{equation} \label{homology01} 
\dim H_1(F_k, H_1(\Delta_k, K))< \infty.
\end{equation}
Note that for every $c \in F_k \cap L$ we have $$[\Delta_k, c] \subseteq [ F_1 \oplus \ldots \oplus F_{k-1}, F_k] = 0,$$ 
hence $c$ centralizes $\Delta_k$ and so $c$ acts trivially on  $H_1(\Delta_k, K)$ i.e. acts as $0$.
Since $F_k \cap L \not= 0$ by (\ref{homology01})  and Lemma \ref{vanishing} applied for $f = c \in (F_k \cap L) \setminus \{ 0 \}$ we deduce that $\dim H_1(\Delta_k, K) < \infty$, as required.

b) By part a) we have that $p_{1, i}(L) = F_1 \oplus F_i$ for every $i \geq 2$. Let $f_1, \ldots , f_{k-1}$ be elements of $F_1$ and let $a_1, \ldots, a_{k-1}$ be elements of $L$ such that $p_{1,i}(a_i) = (f_i, 0) \in F_1 \oplus F_i$. Then the left normed Lie bracket $a =[a_1, \ldots, a_{k-1}]$ has the property that $p_i(a) = 0$ for $i \geq 2$, hence $a \in F_1 \cap L$. On the other hand $a = p_1(a) = [p_1(a_1), \ldots, p_1(a_{k-1})] = [f_1, \ldots, f_{k-1}]$, so $\gamma_{k-1}(F_1) \subseteq L$.

\end{proof}

\begin{lemma} \label{quotient} 
Let $F$ be a non-abelian, finitely generated  free Lie algebra and $N = \gamma_{k-1}(F)$. Then $N^{ab} = N/ [N,N]$ is a finitely generated $\U(Q)$-module that contains a non-trivial cyclic free $\U(Q)$-submodule, where $Q = F/ N$.
\end{lemma}

\begin{proof} Let
$$
{\mathcal R} : 0 \to \U(F)^d \to \U(F) \to K \to 0
$$ be the standard free resolution of the trivial right  $\U(F)$-module $K$, where $d$ is the cardinality of a free generating set of $F$. Apply the functor $ - \otimes_{\U(N)} K$ to  the complex ${\mathcal R}$ and we obtain the complex
$$
{\mathcal U} = {\mathcal R} {\otimes}_{\U(N)} K : 0 \to \U(Q)^d \to \U(Q) \to K \to 0.
$$
Note that ${\mathcal U}$ is exact in dimensions -1 and 0 and
$$
H_1({\mathcal U}) \simeq H_1(N,K) = N/ [N, N].
$$

At this point, observe that as $Q$ is a finitely generated nilpotent Lie algebra, then $Q$ is finite dimensional. By \cite[proposition~6~of~I.2.6]{B} the universal enveloping algebra of any finite dimensional Lie algebra is Noetherian, in particular $\U(Q)$ is Noetherian. Thus as $H_1(N,K)$ is a submodule of $\U(Q)^d$ we deduce that it is finitely generated as $\U(Q)$-module.
We have an exact complex of right $\U(Q)$-modules
\begin{equation} \label{exact1}
0 \to N^{ab} \to  \U(Q)^d \to \U(Q) \to K \to 0.
\end{equation}
Recall that the universal algebra of a Lie algebra does not have zero-divisors.
Observe that since $Q$ is a nilpotent Lie algebra $\U(Q)$ is an Ore ring i.e. for any $a,b \in \U(Q) \setminus \{0 \}$ we have that $a (\U(Q) \setminus \{ 0 \}) \cap b( \U(Q) \setminus \{ 0 \}) \not= \emptyset$.  Indeed by \cite[proposition~10.25]{bookLam} if the universal algebra of a Lie algebra is not Ore it contains  a non-commutative free  algebra, which  is impossible in our case since $Q$ is a nilpotent Lie algebra. Thus there exists the Ore division ring of fractions    $\mathcal F(Q) = \U(Q) (\U(Q) \setminus \{ 0 \})^{-1}$, which is a division $K$-algebra. Since $\mathcal F(Q)$ is a direct limit of the directed set of (left) $\U(Q)$-modules $ \U(Q) r^{-1}$ for $r \in \U(Q) \setminus \{ 0 \}$ and each $\U(Q)$-module $\U(Q) r^{-1}$ is free cyclic, so flat, we deduce that $\mathcal F(Q)$ is a flat (left ) $\U(Q)$-module
and the functor $- \otimes_{\U(Q)} {\mathcal F}(Q)$ is exact. Thus
applying $- \otimes_{\U(Q)} {\mathcal F}(Q)$  to the exact complex (\ref{exact1}) we obtain an exact complex
$$
0 \to N^{ab} \otimes_{\U(Q)} {\mathcal F}(Q) \to {\mathcal F}(Q)^d \to {\mathcal F}(Q) \to 0 \to 0.
$$
Hence 
$$N^{ab} \otimes_{\U(Q)} {\mathcal F}(Q) \simeq	 {\mathcal F}(Q)^{d-1}.$$ Since $d \geq 2$ we are done. Indeed if $m \in N^{ab}$ has the property that $m \otimes 1 \not= 0$ in $N^{ab} \otimes_{\U(Q)} {\mathcal F}(Q)$ then for every $\lambda \in \U(Q) \setminus \{ 0 \}$ we have $m \lambda \not= 0$ in $N^{ab}$ otherwise $m \otimes 1 = m \lambda \otimes \lambda^{-1} = 0 \otimes \lambda^{-1} = 0$, a contradiction.
\end{proof}

\begin{theorem} \label{projection} Let $s \geq 2$ be a natural number and $$L \leq F_1 \oplus \ldots \oplus F_k$$ be a subdirect sum of finitely generated, non-abelian, free Lie algebras $F_1, \ldots, F_k$ such that $L\cap F_i\neq 0$ for $1\leq i\leq k$ and $L$ is of type $FP_s$. Then for every $ 1 \leq i_1 <  i_2 < \ldots < i_s \leq k$ we have that
$p_{i_1, \ldots, i_s}(L) = F_{i_1} \oplus \ldots \oplus F_{i_s}$.
\end{theorem}

\begin{proof} Denote by $N$ the direct sum $\gamma_{k-1}(F_1) \oplus \ldots \oplus \gamma_{k-1}(F_k)$, $N_j = \gamma_{k-1}(F_j)$. By Theorem \ref{nilpotent}  $N$ is an ideal of $L$. Then $L/ N$ is a Lie subalgebra of the nilpotent Lie algebra $Q = Q_1 \oplus \ldots \oplus Q_k$, where $Q_i = F_i / \gamma_{k-1}(F_i)$. This implies that $L/N$ is finitely generated nilpotent thus $\U(L/N)$ is Noetherian.

If $L$ is of type $FP_s$, there is a partial resolution
$$P_s\to P_{s-1}\to\ldots\to P_0\epi K$$ 
with each $P_j$ finitely generated free $\U(L)$-module. Applying $- \otimes_{\U(N)}K$ we get a complex
$$P_s\otimes_{\U(N)}K\to P_{s-1}\otimes_{\U(N)}K\to\ldots\to P_0\otimes_{\U(N)}K$$
whose $i$-th homology is $H_i(N,K)$ for $0 \leq i \leq s-1$ and which consists of finitely generated $\U(L/N)$-modules.  
 By Noetherianess of $\U(L/N)$ we deduce that $H_i(N,K)$ is finitely generated as $\U(L/N)$-module for all $i \leq s$. From now on we assume that $i \leq s$.
Note that using the K\"unneth formula and the fact that $N_j$ is a free Lie algebra (see Lemma \ref{propfree}), so $H_t( N_j,K) = 0$ for $t \geq 2$, we obtain
$$
H_i(N,K) = \oplus_{1 \leq j_1 < \ldots < j_i \leq k} N_{j_1}^{ab} \otimes \ldots \otimes N_{j_i}^{ab},
$$
where $N_j^{ab}$ is the abelianization $N_j/ [N_j, N_j]$ and  $\otimes$ denotes $\otimes_K$.
We deduce that $$W_{j_1, \ldots, j_i} = N_{j_1}^{ab} \otimes \ldots \otimes N_{j_i}^{ab} \hbox{ is a }\U(L/N)-\hbox{direct summand of }H_i(N,K),$$ hence  
is finitely generated.  Note that by Lemma \ref{quotient}  $W_{j_1, \ldots, j_i} $ contains a submodule isomorphic to $$\U(Q_{j_1}) \otimes \ldots \otimes \U(Q_{j_i}) \simeq  \U(Q_{j_1} \oplus \ldots \oplus Q_{j_i}).$$ The action of $L/N$ on $W_{j_1, \ldots, j_i}$ factors trough the action of $q_{j_1, \ldots, j_i}(L/N)$, where $$q_{j_1, \ldots, j_i} : Q \to Q_{j_1} \oplus \ldots \oplus Q_{j_i}$$ is the canonical projection. This means that
$\U(Q_{j_1} \oplus \ldots \oplus Q_{j_i})$ is finitely generated as $\U(q_{j_1, \ldots, j_i}(L/N))$-module. In general, it follows from the proof of \cite[corollary 7.3.9]{Weibel} that whenever $S$ is a Lie subalgebra of a Lie algebra $L$ and $\U(L)$ is finitely generated as $\U(S)$-module, then $S=L$ (in the analogous situation but  in the group ring case we would deduce that one has finite index in the other, but observe that for Lie algebras we do not have the notion of finite index). In our case,  since $q_{j_1, \ldots, j_i}(L/N)$ is a Lie subalgebra of $Q_{j_1} \oplus \ldots \oplus Q_{j_i}$ we deduce that 
$$
q_{j_1, \ldots, j_i}(L/N) = Q_{j_1} \oplus \ldots \oplus Q_{j_i}.
$$
Finally, this implies that
$$
p_{j_1, \ldots, j_i}(L) = F_{j_1} \oplus \ldots \oplus F_{j_i}.
$$
\end{proof}

\begin{theorem}\label{mainfree} Let $F_1,\ldots,F_k$ be finitely generated  free Lie algebras and assume that 
$$L\leq F_1\oplus\ldots\oplus F_k$$
is a Lie algebra of type $FP_k$. Then $L$ is a direct sum of at most $k$ free Lie algebras.
\end{theorem} 

\begin{proof} Observe first that we may assume that $L\cap F_i\neq 0$ for $1\leq i\leq k$ because in other case $L$ would embed in a direct sum of less thank $k$ finitely generated free Lie algebras. 
If all the $F_i$ are non-abelian then by the previous theorem $L =  F_1\oplus\ldots\oplus F_k$. Suppose that some $F_1, \ldots, F_m$ are abelian, hence 1-dimensional and the rest $F_{m+1}, \ldots, F_k$ are non-abelian. Then
the center of $L$ is $Z(L) = L \cap( \oplus_{1 \leq i \leq m} F_i)
$ which is finite dimensional and 
$$L/ Z(L) \leq F_{m+1} \oplus \ldots \oplus F_k.
$$
As $Z(L)$ is finite dimensional and abelian, we deduce that $Z(L)$ is a Lie algebra of type $FP_{\infty}$. Then by Theorem \ref{FPinfty} $L / Z(L)$ is of type $FP_k$ and  by Theorem \ref{projection} 
\begin{equation} \label{quotient1}
L / Z(L) = F_{m+1} \oplus \ldots \oplus F_k.
\end{equation}
We claim that (\ref{quotient1}) implies
\begin{equation} \label{equal01} L \simeq Z(L) \oplus  F_{m+1} \oplus \ldots \oplus F_k
\end{equation}
and $Z(L)$ is a direct sum of at most $m$ abelian free Lie algebras, so $L$ is a direct sum of at most $k$ free Lie algebras. 

To show (\ref{equal01})
consider $f_i \in F_i, f_j  \in F_j$ for $i \not= j$, $\lambda_i, \lambda_j \in Z(L)$. Then there are elements $a_i, a_j \in L$ such that $a_i = (\lambda_i, 0 , \ldots, f_i , \ldots, 0) \in Z(L) \oplus F_{m+1} \oplus \ldots \oplus F_k$ 
and $a_j = (\lambda_j, 0 , \ldots, f_j , \ldots, 0) \in Z(L) \oplus F_{m+1} \oplus \ldots \oplus F_k$. Suppose $i< j$, then $[a_i, a_j] = ([\lambda_i, \lambda_j],0, \ldots, [f_i,0], 0, \ldots, [0,f_j], \ldots, 0) = 0$ in $L$. This completes the proof of  (\ref{equal01}).

\end{proof}

\section{The $n$-$n+1$-$n+2$ Conjecture for Lie algebras}

For the reader's convenience, we recall here the hypothesis of the $1$-$2$-$3$-Conjecture for Lie algebras. We assume that we have short exact sequences of Lie algebras
$$A\rightarrowtail L_1\buildrel{\pi_1}\over\epi Q,$$
$$B\rightarrowtail L_2 \buildrel{\pi_2}\over\epi Q,$$
so that $A$ is finitely generated and $L_1$ and $L_2$ are finitely presented. Note that this implies that $Q$ is finitely presented. 
In the original $1$-$2$-$3$-Conjecture one assumes also that $Q$ is of type $FP_3$. Here we will add an alternative hypothesis and  assume that at least {\bf one} of the following conditions hold:
 \begin{itemize}
\item[a)] $Q$ is of type $FP_3$,

\item[b)] $H_2(A,K)$ is finitely generated as $\U(Q)$-module.
\end{itemize}
In our next result we show that under one of these conditions,
 the fibre sum
$$P= \{ (h_1, h_2) \in L_1 \oplus L_2 \mid  \pi_1(h_1) = \pi_2 (h_2) \} \hbox{ is finitely presented.}
$$

\begin{theorem} \label{123} The $1$-$2$-$3$-Conjecture for Lie algebras holds. Furthermore  the $1$-$2$-$3$-Conjecture holds without assuming that $Q$ is of type $FP_3$ if $H_2(A,K)$ is finitely generated as $\U(Q)$-module.
\end{theorem}

\begin{proof} We divide the proof in three steps.

\medskip
{\bf Claim 1.} {\it If Theorem \ref{123} holds whenever $L_2$ is a free Lie algebra then Theorem \ref{123} holds for any finitely presented Lie algebra $L_2$.} 

\medskip
{\bf Proof of Claim 1:} Let $F$ be a finitely generated free Lie algebra and $\mu : F \to L_2$ a surjective homomorphism of Lie algebras. Composing $\mu$ with the epimorphism $\pi_2 : L_2 \to Q$ gives an epimorphism $\rho : F \to Q$ with kernel $B_0$. Let $P_0$ be the fibre sum of the short exact sequences   $A \rightarrowtail L_1 \to Q$ and $B_0 \rightarrowtail F \epi Q$
i.e.
$$
P_0 = \{ (h_1, h_2) \in L_1 \oplus F \mid  \pi_1(h_1) = \rho (h_2) \}.
$$
By assumption the $1$-$2$-$3$ Conjecture holds in this case i.e. $P_0$ is a finitely presented Lie algebra. Note that $Ker (\mu) \subseteq B_0$ and since $L_2$ is finitely presented as a Lie algebra $Ker (\mu)$ is finitely generated as an ideal of $F$. Then $Ker (\mu)$ is finitely generated as an ideal of $P_0$ and since $P_0 / Ker (\mu) \simeq P$, where $P$ is the fibre sum of the short exact sequences $A \rightarrowtail L_1 \epi Q$ and $B \rightarrowtail L_2 \epi Q$, we deduce that $P$ is a finitely presented Lie algebra and the claim follows.

\medskip

From now we set $F=L_2$ and assume that it  is a finitely generated free Lie algebra.

\medskip
{\bf Claim 2.} {\it There is a finitely presented Lie algebra $\widetilde{P}$ and a short exact sequence of Lie algebras $\Delta\mono\widetilde{P}\epi P$ with $\Delta$ abelian.}

\medskip
{\bf Proof of Claim 2:}
Let $X$ be a free basis for the finitely generated free algebra $L_2=F$ (thus $X$ is finite). Note that $B = Ker (\pi_2)$ is free and let $\{ r_i \}_{i \in I}$ be a free basis of $B$ as a free Lie algebra over $K$. Since $Q \simeq F / B$ is finitely presented, we deduce that $B$ is a finitely generated ideal of $F$. Hence 
\begin{equation} \label{I_0} \hbox{there is a finite subset }I_0 \hbox{ of  }I \hbox{ such that }\{ r_i \}_{i \in I_0} \hbox{ generates }B \hbox{ as an ideal of } F.
\end{equation}

Let $A_0 = \{ a_1, \ldots, a_k \}$ be a finite generating set of the Lie algebra $A$. From the short exact sequence $A\mono L_1\buildrel{\pi_1}\over\epi Q$ we get the following finite presentation  of $L_1$ 
$$
L_1 = \langle X \cup A_0 \mid   R_1 \cup R_2 \cup R_3 \rangle\, 
$$
with
 $$ R_1 = \{ r_i(\underline{x}) - w_i (\underline{a}) \}_{i \in I_0}, R_2 = \{ [a_j, x] - v_{j,x}(\underline{a}) \}_{1 \leq j \leq k, x \in X} \hbox{ and } R_3 = \{ z_j(\underline{a}) \} _{j \in J_0}$$
where $J_0$ is a finite set of indices, $ w_i (\underline{a})  , v_{j,x}(\underline{a}),  z_j(\underline{a})$ are elements of the free Lie algebra with basis $A_0$ and $\pi_1(x)=x$ for any $x\in X$. The same thing happens in $F$: we also have $\pi_2(x)=x$ and this means that the Lie subalgebra of $L_1\oplus F$ generated by $(x,x)$ with $x\in X$ is inside the fibre Lie sum  $P$. In fact that algebra maps isomorphically onto $F$ under the restriction of the second component projection
$$\mu_2:P\to F.$$
The kernel of $\mu_2$ is $\{(a,0)\in P\mid a\in A\}$. Therefore we see that $P$ is isomorphic to the semidirect sum of $A$ by $F$, meaning that $A$ is an ideal. We denote this semidirect sum by $A \leftthreetimes F$. Observe that here we identify $F$ with the algebra generated by $(x,x)$ for $x \in X$.


We define the finitely presented Lie algebra $\widetilde{P}$ by 
$$
\widetilde{P} = \langle  X \cup A_0 \mid  R_2 \cup R_3 \cup R_4\rangle\,$$  with $R_2$ and $R_3$ as before and
$$R_4 = \{  [r_i(\underline{x})  - w_i (\underline{a}), a_j] \}_{i \in I_0, 1 \leq j \leq k.}
$$
We claim that the map induced by $x\mapsto (x,x)$ and $a\mapsto (a,0)$ yields an epimorphism of Lie algebras $$ \delta : \widetilde{P}\epi P.$$ We only have to check that this map is well defined, i.e., that it respects the relators $R_2$, $R_3$ and $R_4$ of $\widetilde{P}$. This is obvious for $R_2$ and $R_3$. For $R_4$ it is a consequence of the fact that in $P$, 
$$\delta(r_i(\underline{x})-w_i(\underline{a}))=
(r_i(\underline{x})-w_i(\underline{a}),r_i(\underline{x}))=(0,r_i(\underline{x}))$$ commutes with $\delta(a_j) = (a_j,0)$.
So if we let $\Delta$ be the kernel of $\delta$ we have the claimed short exact sequence
$$\Delta\mono\widetilde{P}\epi P.$$

We want to show that $\Delta$ is abelian.
Let $\widetilde{A}$ be the ideal of $\widetilde{P}$ generated by  $A_0$. 
Set
$$b_i =  r_i(\underline{x})- w_i(\underline{a}) \in \widetilde{P},\,i\in I_0$$
and let $\widetilde{B}$ be the ideal of $\widetilde{P}$ generated by $\{ b_i \}_{i \in I_0}$. Note that by (\ref{I_0}) for every $i \in I$ we have
$$
r_i(\underline{x})= \sum_{i_0 \in I_0} r_{i_0}(\underline{x}) \circ f_{i_0}(\underline{x}),
$$
where $f_{i_0}(\underline{x})$ belongs to the enveloping algebra  $\U(F)$ and $\circ$ denotes the adjoint action, i.e. $r_{i_0}(\underline{x}) \circ x_1=[r_{i_0}(\underline{x}) , x_1]$. As a consequence,
$$  r_{i_0}(\underline{x}) \circ x_1 x_2 =(r_{i_0}(\underline{x}) \circ x_1)\circ x_2=  [[r_{i_0}(\underline{x}), x_1], x_2].$$  Then define
$$
b_i = \sum_{i_0 \in I_0} b_{i_0} \circ f_{i_0}(\underline{x}) =
\sum_{i_0 \in I_0} r_{i_0}(\underline{x}) \circ f_{i_0}(\underline{x})
 - \sum_{i_0 \in I_0} w_{i_0}(\underline{a}) \circ f_{i_0}(\underline{x}) =  $$ $$ r_i(\underline{x}) -  \sum_{i_0 \in I_0} w_{i_0}(\underline{a}) \circ f_{i_0}(\underline{x}) 
   \in 
r_i(\underline{x}) + \widetilde{A}  \ \  \hbox{ for } i \in I \setminus I_0
$$
and set $w_i(\underline{a}) = r_i(\underline{x}) - b_i \in \widetilde{A}$ for $i \in I \setminus I_0$, hence
$$
b_i =  r_i(\underline{x}) - w_i(\underline{a}) \in \widetilde{B} \ \  \hbox{ for } i \in I.$$

Let $B_1$ be the Lie subalgebra  of $\widetilde{B}$ generated by $\{ b_i \}_{i \in I}$.  When projected from $\widetilde{P} $ to $F$ via the Lie algebra homomorphism that is the identity on $X$ and sends each element of $A_0$ to $0$, 
$B_1$ maps onto the free Lie algebra $B = \langle \{ r_i \}_{i \in I}\rangle$ with a free basis $\{ r_i \}_{i \in I}$, so  $B_1 \simeq B$ i.e. $$B_1 \hbox{ is a free Lie algebra with a free basis  }\{ b_i \}_{i \in I}.$$ Thus the ideal $M=\widetilde{A}+\widetilde{B}$ of $\widetilde{P}$ is a semidirect sum of Lie algebras  $\widetilde{A} \leftthreetimes B_1$.

We claim that the relations $R_2$ and $R_4$ imply that \begin{equation} \label{commutator} [\widetilde{A}, \widetilde{B}] = 0. \end{equation} Indeed let $f(\underline{x}) \in F^m \setminus F^{m-1} \subseteq  \U(F)$, where $F^{-1} = \emptyset$, $F^0 = 1_{\U(F)}$, $F^1 = F$ and $F^{i + 1 }  = F^i F$ in $\U(F)$. We prove by induction on $m$ that \begin{equation} \label{commutator2}
[b_i \circ f(\underline{x}) , a_j] = 0 \hbox{ in }\widetilde{P} \hbox{ for } i \in I_0, 1 \leq j \leq k.\end{equation} Note that if $m = 0$ there is nothing to prove as the set $R_4$ is included in the relations of $\widetilde{P}$. If $m \geq 1$ we have a decomposition $f(\underline{x}) = f_0(\underline{x}) x$ for some $x \in X$, $f_0(\underline{x}) \in F^{m-1} \setminus F^{m-2}$. Then   we have by induction that $[b_i \circ f_0(\underline{x}), \langle A_0 \rangle ] = 0$ in $\widetilde{P}$ for $i \in I_0$, where $\langle A_0 \rangle$ is the Lie subalgebra of $\widetilde{P}$ generated by $A_0$. Furthermore for $1 \leq j \leq k$, $i \in I_0$ we have (using the Jacobi identity)
$$
[b_i \circ f(\underline{x}), a_j] = [b_i \circ f_0(\underline{x})x, a_j] =
[b_i \circ f_0(\underline{x}), a_j] \circ x - [b_i \circ f_0(\underline{x}), a_j \circ x ] =$$ $$ - [b_i \circ f_0(\underline{x}), a_j \circ x ] = - [b_i \circ f_0(\underline{x}), v_{j,x}(\underline{a}) ]  \in  [b_i \circ f_0(\underline{x}), \langle A_0 \rangle] = 0 \hbox{ in } \widetilde{P}.
$$
Using (\ref{commutator2}) we show
\begin{equation} \label{comm3}
[b_i \circ f(\underline{x}), a_j \circ g(\underline{x})] = 0 \hbox{ in }\widetilde{P} \hbox{ for } i \in I_0, 1 \leq j \leq k, f(\underline{x}),g(\underline{x}) \in \U(F)
\end{equation}
It suffices to show (\ref{comm3}) for $g \in F^s \setminus F^{s+1} \subset \U(F)$ and we induct on $s \geq 0$, the case when $s = 0$ means that $g(\underline{x})$ is a scalar i.e. belongs to $K$. The calculation is similar to the above calculation (using Jacobi identity), for completeness we include the details. Suppose $s \geq 1$, then $g(\underline{x}) = g_0(\underline{x}) x$ for some $x \in X,g_0(\underline{x}) \in F^{s-1} \setminus F^{s-2}$. Then by induction on $s$ we have
$$
[b_i \circ f(\underline{x}), a_j \circ g_0(\underline{x})] = 0 = [b_i \circ f(\underline{x}) x, a_j  \circ g_0(\underline{x})]
$$ and hence
$$
[b_i \circ f(\underline{x}), a_j \circ g(\underline{x})] =
[b_i \circ f(\underline{x}), (a_j \circ g_0(\underline{x})) \circ x] =$$
$$
[b_i \circ f(\underline{x}), a_j \circ g_0(\underline{x})] \circ x - 
[b_i \circ f(\underline{x}) x, a_j \circ g_0(\underline{x})]  = 0.
$$
Since $\widetilde{B}$ is the ideal of $\widetilde{P}$ generated by $\{ b_i \}_{i \in I_0}$, (\ref{comm3}) implies (\ref{commutator}).

Observe that under the projection $\delta : \widetilde{P}\epi P$, $\widetilde{A}$ is mapped onto $\{(a,0) \mid a\in A\}$ and $\widetilde{B}$ is mapped onto $\{(0,b) \mid b\in B\}$.  Recall $\Delta = Ker (\delta)$. 
Since $\widetilde{P} / \widetilde{A} \simeq P / A$ and $\widetilde{P}/ \widetilde{B} \simeq P / B$ we deduce that $\Delta \subseteq \widetilde{A} \cap \widetilde{B}.$ On the other hand $A \cap B = 0 $ in $P$, hence $\widetilde{A} \cap \widetilde{B} \subseteq \Delta.$
Thus, 
$$\Delta = \widetilde{A} \cap \widetilde{B}.$$
Note that $\widetilde{A} / \Delta \simeq A$ and $ \widetilde{B} / \Delta \simeq B$. Also, as $[\widetilde{A},\widetilde{B}]=0$ we deduce that $\Delta$ is abelian as we wanted to show.

\medskip
{\bf Claim 3.} {\it $\Delta$ is finitely generated as $\U(Q)$-module.}

\medskip
{\bf Proof of Claim 3:}
  We split the proof in two parts according to whether we have hypothesis a) or b) from the beginning of the proof. 

Suppose first that condition a) holds.  Consider the central short exact sequence of Lie algebras 
$$
 \Delta \rightarrowtail \widetilde{B} \epi B. 
$$
Since $B$ is free, $\widetilde{B} \simeq  \Delta \leftthreetimes B$. Identifying $\widetilde{B}$  with $ \Delta \leftthreetimes B$ and using that $\Delta$ is central in $\widetilde{B}$ we obtain that $[\widetilde{B},\widetilde{B}] = [B,B]$, hence $\Delta \cap [\widetilde{B},\widetilde{B}] = \Delta \cap [B,B] \subseteq \Delta \cap B =0 $.
This means that we have a short exact sequence
\begin{equation}\label{Bab}\Delta\mono\widetilde{B}^{ab}\epi B^{ab}.\end{equation}
By the same argument as at the beginning of the proof of Lemma \ref{quotient}, using the isomorphism $F/B \simeq Q$ one can see that there is an exact sequence
$$0\to B^{ab}\to\U(Q)^d\to\U(Q)\to K\to 0$$
where $d$ is the rank of the free Lie algebra $F$.
As $Q$ is of type $FP_3$ we deduce by the Lie algebra version of the proof of \cite[VIII 4.3]{Brown} that $B^{ab}$ is finitely presented as $\U(Q)$-module i.e. is $FP_1$ as $\U(Q)$-module.
So using the short exact sequence (\ref{Bab}) and the Lie algebra version of \cite[proposition 1.4]{Bieribook} we get that $\Delta$ is finitely generated as $\U(Q)$-module because the $\U(Q)$-modules  $B^{ab}$ and $\widetilde{B}^{ab}$ are  finitely presented and finitely generated respectively.

b) We assume now that $H_2(A,K)$ is finitely generated as $\U(Q)$-module. 
Now we choose the elements of $R_3$ in such a way that  the canonical projection $\widetilde{A} \to A$ induces an isomorphism $\widetilde{A}^{ab} \simeq A^{ab}$ (note that here we use that $A^{ab} $ is finite dimensional). 
Then since $\widetilde{A}/ \Delta \simeq A$ we get  $0 = Ker (\widetilde{A}^{ab} \to A^{ab}) = (\Delta + [\widetilde{A}, \widetilde{A}])/ [\widetilde{A}, \widetilde{A}]$, hence $\Delta \subseteq [\widetilde{A}, \widetilde{A}] \subseteq [M,M]$, where $M = \widetilde{A} + \widetilde{B} = \widetilde{A} \leftthreetimes B_1$. Consider the short exact sequence of Lie algebras
$$
 \Delta \rightarrowtail M \epi A \oplus B 
$$
Since $\Delta \subseteq [M,M]$ and $\Delta$ is central in $M$ we deduce by Lemma \ref{schur} that $\Delta$ is a surjective image of $H_2(A \oplus B,K)$. Therefore, 
 if $H_2(A \oplus B,K)$ is finitely generated as $\U(Q)$-module then $\Delta$ is finitely generated as $\U(Q)$-module.
By the K\"unneth formula $$H_2(A \oplus B,K) = H_2(A,K) \oplus H_2(B,K) \oplus H_1(A,K) \otimes H_1(B,K).$$
Observe that  since $B$ is a free Lie algebra $H_2(B,K) = 0$. 
Moreover $H_1(B,K) = B^{ab} = B/ [B,B]$ is a finitely generated $\U(Q)$-module and $H_1(A,K) = A/ [A,A]$ is finite dimensional thus our hypothesis that   $H_2(A,K)$ is a finitely generated $\U(Q)$-module implies that also $H_2(A \oplus B,K)$ is. This completes the proof of Claim 3.

\medskip

Finally we can complete the proof. At this point, we have $$\Delta\mono\widetilde{P}\epi P$$
with $\widetilde{P}$ finitely presented and $\Delta$ abelian and finitely generated as $\U(Q)$-module thus also as $\U(P)$-module. Then adding a finite generating set of $\Delta$ to the presentation of $\widetilde{P}$ we get a finite presentation of $P$. 
\end{proof}

\begin{remark} Observe that if we assume that $Q$ is finitely generated and nilpotent then it is finite dimensional, so is of type $FP_\infty$. In particular $Q$ is of type $\FP_3$ and we may apply Theorem \ref{123}. 
\end{remark}

\begin{theorem} \label{fin-pres} Let $L \leq L_1 \oplus \ldots \oplus L_k$ be a subdirect Lie sum with $L \cap L_i \not= 0$ and $L_i$ finitely presented for $1\leq i\leq k$. Assume that $p_{i,j} (L) = L_i \oplus L_j$ for all $1 \leq i < j \leq k$, where  $p_{i,j} : L_1 \oplus \ldots \oplus L_k \to L_i \oplus L_j$ is the canonical projection. Then $L$ is a finitely presented Lie algebra.
\end{theorem}

\begin{proof} 1. We claim first that $\gamma_{k-1}(L_j) \subseteq L$ for every $j$, assume from now on that $j = k$. Indeed let $l_1, \ldots, l_{k-1}$ be elements of $L_k$. Let $h_i$ be an element of $ L$ such that $p_{i, k}(h_i) = l_i \in L_k \subset L_i \oplus L_k$ for $1 \leq i \leq k-1$.
Then
$$(0,\ldots,0,[l_1, \ldots, l_{k-1}])=[h_1, \ldots, h_{k-1}]\in L.$$ 
 As the elements of the form $[l_1, \ldots, l_{k-1}] $ generate $\gamma_{k-1}(L_k) $ as a Lie algebra the claim follows.

2. We prove the theorem by induction on $k \geq 2$. When $k = 2$ there is nothing to prove so we assume $k \geq 3$. Let $Q=L_k/L\cap L_k.$ 
By the first paragraph $Q$ is a finitely generated nilpotent Lie algebra, hence of finite dimension. Moreover there is a short exact sequence
\begin{equation} \label{ses01} L \cap (L_1 \oplus \ldots \oplus L_{k-1}) \rightarrowtail  p_{1,2, \ldots , k-1} (L) \epi Q\end{equation}
where the right hand map is given by $(l_1,\ldots,l_{k-1})\mapsto q+L\cap L_k$ for $q\in L_k$ such that $(l_1,\ldots,l_{k-1},q)\in L$ (one easily checks that this is well defined). Then we see that $L$ is the fibre sum  of the short exact sequence (\ref{ses01}) and
$$L \cap L_k \mono L_k \epi Q.$$ 
Moreover, $p_{1,2, \ldots , k-1} (L) \leq L_1 \oplus \ldots \oplus L_{k-1}$ is a subdirect Lie sum that maps epimorphically onto $L_i \oplus L_j$ for $ 1 \leq i < j \leq k-1$. Then by induction   $p_{1,2, \ldots , k-1} (L)$ is finitely presented. If we show that  $L \cap (L_1 \oplus \ldots \oplus L_{k-1})$ is a finitely generated  Lie algebra we can apply   Theorem \ref{123} to deduce that $L$ is finitely presented.
To do that observe first  that 
\begin{equation} \label{subdirect1} \gamma_{k-1}(L_1) \oplus \ldots \oplus \gamma_{k-1}(L_{k-1}) \subseteq L \cap (L_1 \oplus \ldots \oplus L_{k-1}).
\end{equation}  Furthermore since $p_{i,k}(L) = L_i \oplus L_k$ we have that
$$p_i(L \cap (L_1 \oplus \ldots \oplus L_{k-1})) = L_i \hbox{ for }1 \leq i \leq k-1.$$
Now, let $Y_i$ be a finite subset of $L \cap (L_1
\oplus \ldots \oplus L_{k-1})$ such that $p_i(Y_i)$ generates $L_i$ and let
$Z_i$ be a finite subset of $\gamma_{k-1}(L_i)$ that generates it as an
ideal of $L_i$ (we may choose such a set because $L_i / \gamma_{k-1}(L_i)$ is finitely presented). Then the Lie algebra generated by $\cup_{1 \leq i \leq k-1}(Y_i \cup Z_i)$ contains $\gamma_{k-1} (L_1) \oplus \ldots \oplus \gamma_{k-1}(L_{k-1})$. Since $S = (L \cap  (L_1 \oplus \ldots \oplus L_{k-1})) / \gamma_{k-1} (L_1) \oplus \ldots \oplus \gamma_{k-1}(L_{k-1})$ is a Lie subalgebra  of the finitely generated nilpotent Lie algebra 
$(L_1/ \gamma_{k-1} (L_1)) \oplus \ldots \oplus (L_{k-1}/ \gamma_{k-1} (L_{k-1}))$ we deduce that $S$ is a finitely generated Lie algebra (it is even finite dimensional). Let $T$ be a finite subset of $L \cap (L_1 \oplus \ldots \oplus L_{k-1})$ whose image in $S$ is a finite generating set of $S$. Then $T \cup (\cup_{1 \leq i \leq k-1}(Y_i \cup Z_i))$ is a finite generating set of $L \cap (L_1 \oplus \ldots \oplus L_{k-1})$.
\end{proof} 
 Theorem \ref{projection} and Theorem \ref{fin-pres} imply immediately the following result. 
\begin{corollary} Let $L$ be a subdirect Lie sum of $L_1 \oplus \ldots \oplus L_k$ with $L \cap L_i \not= 0$ for $1\leq i\leq k$, where each $L_i$ is a non-abelian free Lie algebra. Then $L$ is finitely presented if and only if
$p_{i,j}(L) = L_i \oplus L_j$ for all $1 \leq i < j \leq k$. 
\end{corollary}

\section{On the homological $1$-$2$-$3$ Conjecture for Lie algebras}

In this section we show that homological versions of Theorems \ref{123} and \ref{fin-pres} hold true, where in the homological version of Theorem \ref{123} we assume further that $Q$ is finitely presented.  Basically, to do that we relax the hypothesis of the finite presentability to being of type $FP_2$.

\begin{theorem} The homological $1$-$2$-$3$ Conjecture for Lie algebras holds if in addition $Q$ is a finitely presented Lie algebra. Furthermore the condition that $Q$ is $FP_3$ could be substituted with $H_2(A,K)$ is finitely generated as $\U(Q)$-module.
\end{theorem} 

\begin{proof} Let $ A \rightarrowtail L_1 \epi Q $ be a short exact sequence of Lie algebras with $A$ finitely generated, $L_1$ of type $FP_2$  and $Q$ finitely presented and let $B \rightarrowtail L_2 \epi Q$ be a short exact sequence of Lie algebras with $L_2$ of type $FP_2$. Denote by $P$ the fibre Lie sum of these sequences.

The proof that $P$ is $FP_2$ goes along similar lines as the proof of Theorem \ref{123} in the sense that
we begin by showing that it suffices to consider the case when $L_2 = F$ is a free Lie algebra with a finite free basis $X$ and then construct a suitable Lie algebra $\widetilde{P}_1$ which is finitely presented and that we will use to deduce that $P$ is $FP_2$. 

To see that we may assume that $L_2$ is free, observe that with the same notation as at the beginning of Theorem \ref{123}, we have short exact sequences
$$Ker(\mu)\mono F\epi L_2\, \hbox{ and }Ker(\mu)\mono P_0\epi P,$$
i.e. $\mu : F \to L_2$ is an epimorphism of Lie algebras, $F$ is a finitely generated free Lie algebra and  $P_0$ is the fibre sum  of $ A \rightarrowtail L_1 \epi Q $ and $B_0 \rightarrowtail F \epi Q$, where the epimorphism $F \epi Q$ is $\pi_2\circ \mu$. 
Let $F_0$ be a finitely generated free Lie algebra with an epimorphism  $\tau_0:F_0\epi P_0$. 
As we may assume that $P_0$ is $FP_2$, $Ker(\tau_0)^{ab}$ is finitely generated as $\U(P_0)$-module. But observe that composing $\tau_0$ with the epimorphism  $P_0\epi P$ yields an epimorphism $\tau:F_0\epi P$ and a short exact sequence
$$Ker(\tau_0)\mono Ker(\tau)\epi Ker(\mu). $$
Moreover, the short exact sequence $Ker(\mu)\mono F\epi L_2$ together with the fact that $L_2$ is $FP_2$ imply that also $Ker(\mu)^{ab}$ is finitely generated as $\U(L_2)$-module. From all this one deduces that \begin{equation} \label{fin-gen-1} Ker(\tau)^{ab} \hbox{ is finitely generated as }\U(P)-\hbox{module},
\end{equation} because there is a short exact sequence
$$M\mono Ker(\tau)^{ab}\epi Ker(\mu)^{ab}$$
with $M$ a quotient of $Ker(\tau_0)^{ab}$.
Note that (\ref{fin-gen-1}) implies that $P$ is of type $FP_2$.

From now on we assume that $L_2$ is a free Lie algebra $F$ with a finite free basis $X$. Let
$$Q = \langle X \mid  \{ r_i \}_{i \in I_0}  \rangle$$
be a finite presentation of the Lie algebra $Q$.
Then we have a presentation of $L_1$ 
$$
L_1 = \langle X \cup A_0 \mid   R_1 \cup R_2 \cup R_3 \rangle
$$
with  $A_0$ and $X$  both finite and
 $$ R_1 = \{ r_i(\underline{x}) - w_i (\underline{a}) \}_{i \in I_0}, R_2 = \{ [a_j, x] - v_{j,x}(\underline{a}) \}_{1 \leq j \leq k, x \in X} \hbox{ and } R_3 = \{ z_j(\underline{a}) \} _{j \in J}$$
where $I_0$ and $J$ are set of indices with $I_0$ finite and $J$ possibly infinite and $w_i (\underline{a}) $, $ v_{j,x}(\underline{a})$, $  z_j(\underline{a})$ are elements of the free Lie algebra with a free basis $A_0$. Let $R$ be the ideal of the free Lie algebra $F(X\cup A_0)$ with a free basis $X \cup A_0$ generated by $R_1 \cup R_2 \cup R_3$. Then we have a short exact sequence
$$R\mono F(X\cup A_0)\epi L_1$$
 and since $L_1$ is of type $FP_2$, $R/[R,R]$ is finitely generated as $\U(L_1)$-module. This means that
  there is a finite subset $J_0$ of $J$ such that 
\begin{equation} \label{FP_2condition}
R = R_0 + [R,R],
\end{equation}
where $R_0$ is the ideal of the free Lie algebra $F(X \cup A_0)$ with a free basis $X \cup A_0$ generated by $R_1 \cup R_2 \cup R_{3,0}$, where $R_{3,0} = \{ z_j(\underline{a}) \} _{j \in J_0} \subseteq R_3$.

Consider the Lie algebra $\widetilde{L}_1 = F(X \cup A_0) / R_0$. There is a natural projection
$$
\pi : \widetilde{L}_1 = F(X \cup A_0)/ R_0 \to L_1 = F(X \cup A_0) / R
$$ with kernel $S = R / R_0$. By (\ref{FP_2condition}) we have
\begin{equation} \label{kernelFP2}
S = [S,S].
\end{equation}
Since $R_1 \cup R_2$ are relations in $\widetilde{L}_1$ we deduce that the Lie subalgebra $\widetilde{A}_1$ of $\widetilde{L}_1$ generated by the elements of $A_0$ is an ideal such that $\widetilde{L}_1 / \widetilde{A}_1 \simeq Q$. Thus there is a short exact sequence of Lie algebras
\begin{equation}\label{tildeA1}
\widetilde{A}_1 \rightarrowtail \widetilde{L}_1 \buildrel{\widetilde{\pi}_1}\over\epi Q
\end{equation}
with both $\widetilde{L}_1$ and $Q$ finitely presented.
Denote by $\widetilde{P}_1$ the fibre Lie sum of the short exact sequences of Lie algebras  (\ref{tildeA1}) and 
$$B \rightarrowtail F \buildrel{\pi_2}\over\epi Q.$$
 As either $Q$ is $FP_3$ or $H_2(A, K)$ is finitely generated as $\U(Q)$-module, Theorem \ref{123} implies that $\widetilde{P}_1$ is finitely presented. We have
 $$\widetilde{P}_1=\{(h_1,h_2)\in \widetilde{L}_1\oplus F\mid \widetilde{\pi}_1(h_1)=\pi_2(h_2)\}$$
 and if $P$ is the fibre Lie sum of the original short exact sequences $A \rightarrowtail L_1 \buildrel{\pi_1}\over\epi Q$
and $B \rightarrowtail F \buildrel{\pi_2}\over\epi Q$, then
$$P=\{(h_1,h_2)\in L_1\oplus F\mid\pi_1(h_1)=\pi_2(h_2)\}.$$
This means that the map $(h_1,h_2)\mapsto (\pi(h_1),h_2)$ is a well defined epimorphism
$$\mu:\widetilde{P}_1\epi P$$
with kernel $Ker(\mu)=S.$
 
 It follows from the proof of Theorem \ref{123} that  $\widetilde{P}_1=F(X \cup A_0) / S_1$ for some ideal $S_1$ of $F(X \cup A_0)$ thus from the above considerations we deduce that $P = F(X \cup A_0) / S_2$ for  $S_2$ ideal  of $F(X \cup A_0)$ such that  $S_1 \subseteq S_2$ and $S_2/ S_1 = S = [S,S]$. Hence
$$
S_2 = S_1 + [S_2, S_2],
$$
so $S_2/ [S_2, S_2]$ is an epimorphic image of $S_1/ [S_1, S_1]$.
Since $\widetilde{P}_1$ is of type $FP_2$ we deduce that $S_1/ [S_1, S_1]$ is finitely generated as $\U(\widetilde{P}_1)$-module. Thus its epimorphic image $S_2/ [S_2, S_2]$ is finitely generated as $\U(P)$-module, hence $P$ is of type $FP_2$ as claimed.
\end{proof}

\begin{theorem} \label{FP2-} Let $$L \leq L_1 \oplus \ldots \oplus L_k$$  be a subdirect Lie sum with $L \cap L_i \not= 0$ and $L_i$ a Lie algebra of type $FP_2$ for $1\leq i\leq k$. Assume that $p_{i,j} (L) = L_i \oplus L_j$ for all $1 \leq i < j \leq k$, where  $p_{i,j} : L_1 \oplus \ldots \oplus L_k \to L_i \oplus L_j$ is the canonical projection. Then $L$ is a  Lie algebra of type $FP_2$.
\end{theorem}

The proof of Theorem \ref{FP2-} can be obtained from the proof of Theorem \ref{fin-pres} by substituting everywhere finitely presented with type $FP_2$ and substituting the $1$-$2$-$3$ Conjecture for Lie algebras with the homological $1$-$2$-$3$ Conjecture for Lie algebras. 

Theorem \ref{projection} and Theorem \ref{FP2-} imply immediately the following result.

\begin{theorem} Let  $$L \leq L_1 \oplus \ldots \oplus L_k$$ be a subdirect Lie sum 
 with $L \cap L_i \not= 0$ and $L_i$ free, non-abelian, Lie algebra for all $i$. Then $L$ is finitely presented as a Lie algebra if and only if $L$ is of type $FP_2$.
 \end{theorem}

\section{On a homological $n-n+1-n+2$ Conjecture for Lie algebras}

Although the homotopical $s$-$s+1$-$s+2$ Conjecture for groups remains open in general, Kuckuck has obtained some partial results in \cite{Beno}. 
More precisely, he shows that the  homotopical $s$-$s+1$-$s+2$ Conjecture for groups holds if {\bf one} of the following assumptions holds:
\begin{itemize}
\item[a)] if the second sequence splits; 

\item[b)] if the conjecture holds whenever $L_2$ is free. 
\end{itemize} 

The proof of these two reductions is based in the following Theorem which is one of the main results in \cite{Beno}. Recall that a group is of homotopical type $F_s$ for $s \geq 2$ if it is finitely presented and  of type $FP_s$. And a group is of type  $F_1$ if it is finitely generated.

\begin{theorem} \label{kuckuck1} \cite{Beno} Assume that we have a commutative diagram of groups with exact rows
$$\begin{tikzcd}A \arrow[r, tail]\arrow[d,"id_A"] &B_0 \arrow[r,twoheadrightarrow]\arrow[d,"\theta"]& C_0\arrow[d,"\nu"]\\
A \arrow[r, tail] &B \arrow[r,twoheadrightarrow]& C\\
\end{tikzcd}$$
where $A$ is $F_s$, $B_0$ is $F_{s+1}$ and $C$ is $F_{s+1}$ for some $s \geq 2$. Then $B$ has type $F_{s+1}$.
\end{theorem}

The proof of Theorem \ref{kuckuck1} is homotopical and uses the Borel construction and the theory of stacks. Recently in \cite{Francismar} a
 homological version of the above Theorem was proved based on non-homotopical methods. Here we show a Lie algebra version of Theorem \ref{kuckuck1} when $\theta$ is either a monomorphism or an epimorphism. The starting point is the following result.

\begin{proposition} \label{homology1}
Let  $B$ be a Lie algebra with an ideal $A$ such that $A$ is of type $FP_s$ and $C = B/A$ is a Lie algebra of type $FP_{s+1}$.  Then
$B$ is of type $FP_{s+1}$ if and only if the map
$$
d_{s+1,0}^{s+1} : H_{s+1} (C, \prod \U(C)) \to H_0(C, H_s(A, \prod \U(B)))
$$
is surjective for any  direct product, where $d_{s+1,0}^{s+1}$ is the differential arising from  the Lyndon-Hoschild-Serre spectral sequence 
$$
E_{p,q}^2 = H_p(C, H_q(A, \prod \U(B)))\Rightarrow H_{p+q}(B, \prod \U(B)).$$
\end{proposition}

\begin{proof} Using that  $A$ is a Lie algebra of type  $FP_s$ and that $\U(B)$ is free as $\U(A)$-module we obtain by Theorem \ref{LieBieri} 
$$H_q(A, \prod \U(B)) = \prod H_q(A, \U(B)) = \prod 0 = 0 \hbox{ for }1 \leq q \leq s-1 ,$$ hence
\begin{equation} \label{spectral}
E_{p,q}^{j} = 0 \hbox{ for }1 \leq q \leq s-1, 2 \leq j \leq \infty. \end{equation}
Since $A$ is finitely generated by Theorem \ref{LieBieri} we get  $$H_0(A, \prod \U(B)) \simeq \prod H_0(A, \U(B)) = \prod \U(B/A) = \prod \U(C).$$ 
Since $C$ is of type $FP_{s+1}$ 
$$ \label{new01} 
E_{p,0}^2 = H_p(C, \prod \U(C)) \simeq \prod  H_p(C, \U(C)) = \prod 0 = 0 \hbox{ for } 1 \leq p \leq s,$$
so 
\begin{equation} \label{novo001}  E^{j}_{p,0} = \hbox{ for } 1 \leq p \leq s, 2 \leq j \leq \infty.
\end{equation}
Then by (\ref{spectral}) and (\ref{novo001})
\begin{equation} \label{novo}
E_{p, s-p}^j = 0 \hbox{ for every } 0<p\leq s, 2 \leq j \leq \infty.
\end{equation}
 By (\ref{spectral}) 
$$
E_{2,s-1}^2 = E_{3,s-2}^3 = \ldots = E_{i, s+1-i}^i = \ldots = E_{s,1}^s = 0,
$$
hence 
\begin{equation} \label{new02}
d_{i, s+1 - i}^i : E_{i, s+1-i}^i \to E_{0,s}^{i}  \hbox{ is the zero map for } 2 \leq i \leq s.\end{equation}
On the other hand the same thing happens for $i \geq s+2$ because then $E_{i, s+1-i}^i=0$ since $s+1 - i < 0$.
Note that the differential
\begin{equation} \label{new03}
d_{0,s}^i : E_{0,s}^i \to E_{-i, s-1 + i}^i = 0 \hbox{ is always zero.}
\end{equation}
Then (\ref{new02}) and (\ref{new03}) imply that $ E_{0,s}^{i} = E_{0,s}^{i+1} $ for every $2 \leq i \leq s $ and $i \geq s+2$, so
\begin{equation} \label{new04}
E^{2}_{0,s} = E_{0,s}^{s+1}, E_{0,s}^{s+2} = E^{\infty}_{0,s}
\end{equation}
Since $A$ and $C$ are Lie algebras of type $FP_s$ we have that $B$ is of type $FP_s$. Then by Theorem \ref{LieBieri} 
\begin{equation} \label{new05} \hbox{the Lie algebra }B \hbox{ is of type }FP_{s+1}\hbox{ if and only if }H_s(B, \prod \U(B)) = 0.\end{equation} By the convergence of the spectral sequence  and (\ref{novo})
\begin{equation} \label{new06}
H_s(B, \prod \U(B))  \simeq  E^{\infty}_{0,s}. 
\end{equation}
By (\ref{new04}), (\ref{new05}) and (\ref{new06}) 
$$\hbox{the Lie algebra }B \hbox{ is of type }FP_{s+1}\hbox{ if and only if }E_{0,s}^{s+2} = 0.$$ 
Since $d^{s+1}_{0,s} : E_{0,s}^{s+1} \to E^{s+1}_{-s-1, 2s} = 0$ is the zero map, $E_{0,s}^{s+2}$ is the cokernel of the differential $d_{s+1,0}^{s+1}  : E_{s+1,0}^{s+1} \to E_{0,s}^{s+1}$. Thus $E_{0,s}^{s+2} = 0$ if and only if  $d_{s+1,0}^{s+1}$ is surjective.

Finally note that by (\ref{spectral}) all differentials that leave $E_{s+1,0}^j$ , for $2 \leq j \leq s$ are zero. Comparing bidegrees we  get that all differentials that enter $E_{s+1,0}^j$ , for $j \geq 2$ are zero. Hence $$E_{s+1,0}^{s+1} = E_{s+1,0}^{2} = H_{s+1}(C, H_0(A, \prod \U(B))) =$$ $$ H_{s+1}(C, \prod H_0(A, \U(B))) = H_{s+1}(C, \prod \U(C)).$$ By (\ref{new04})
$$
E^{s+1}_{0,s} = E^2_{0,s} =  H_0(C, H_s(A, \prod \U(B))).
$$
Hence the domain and target of $d_{s+1,0}^{s+1}$ are as described in the statement of the result.
\end{proof}

\begin{theorem} \label{homologicalkuckuck} Assume that we have a commutative diagram of Lie algebras with exact rows
$$\begin{tikzcd}A \arrow[r, tail]\arrow[d,"id_A"] &B_0 \arrow[r,twoheadrightarrow]\arrow[d,"\theta"]& C_0\arrow[d,"\nu"]\\
A \arrow[r, tail] &B\arrow[r,twoheadrightarrow]& C\\
\end{tikzcd}$$
where $A$ is $F_s$, $B_0$ is $F_{s+1}$ and $C$ is $F_{s+1}$ for some $s \geq 1$. Assume that at least one of the following conditions holds true :
\begin{itemize}
\item[a)] $\theta$ is a monomorphism;

\item[b)] $\theta$ is an epimorphism.
\end{itemize}
Then $B$ has type $FP_{s+1}$.
\end{theorem}

\begin{proof}  

\noindent{\bf{Case a): when $\theta$ is mono.}} We view $B_0$ as a Lie subalgebra of $B$, so $\U(B)$ is a free $\U(B_0)$-module. Note that for any direct product $ \prod \U(B)$, repeating the proof of Proposition \ref{homology1} for  the spectral sequence associated  to the short exact sequence $A \rightarrowtail B_0 \twoheadrightarrow C_0$, 
 $$
\widetilde{E}_{p,q}^2 = H_p(C_0, H_q(A, \prod \U(B)))\Rightarrow H_{p+q}(B_0, \prod \U(B))$$ we obtain that the differential 
$$
\widetilde{d}_{s+1,0}^{s+1} : H_{s+1} (C_0, \prod \U (C)) \to H_0(C_0, H_s(A, \prod \U(B)))
$$
is surjective if and only if  $H_s(B_0, \prod \U(B)) = 0$. Since $B_0$ is $FP_{s+1}$ we have  $$H_s(B_0, \prod \U(B)) =   \prod H_s(B_0, \U(B))  = \prod 0 = 0,$$ hence
\begin{equation} \label{surj01} 
\widetilde{d}_{s+1,0}^{s+1}  \hbox{ is surjective.}
\end{equation}
\noindent
 Furthermore $\nu$ induces a commutative diagram
$$\begin{tikzcd} H_{s+1} (C_0, \prod \U (C) )\arrow[r,"\widetilde{d}_{s+1,0}^{s+1}"]  \arrow[d,"\mu_1"]& H_0(C_0, H_s(A, \prod \U(B)))\arrow[d,"\mu_2"]  \\
H_{s+1} (C, \prod \U(C))   \arrow[r,"d_{s+1,0}^{s+1}"] &  H_0(C, H_s(A, \prod \U(B)))  \end{tikzcd},
$$ 
\noindent
where ${d}_{s+1,0}^{s+1} $ is the differential of the LHS spectral sequence $$
E_{p,q}^2 = H_p(C, H_q(A, \prod \U(B)))\Rightarrow H_{p+q}(B, \prod \U(B)).$$ 
associated to the short exact sequence $A \rightarrowtail B \twoheadrightarrow C$. The vertical maps are induced by $\nu$. Since $\widetilde{d}_{s+1,0}^{s+1}$ and $\mu_2$ are both surjective maps the differential ${d}_{s+1,0}^{s+1}$ is surjective. Then by Proposition \ref{homology1}, $B$ is a Lie algebra of type $FP_{s+1}$.
\

\noindent{\bf{Case b): when $\theta$ is epi.}} Set $T = Ker(\theta)$ and consider the short exact sequence of left $\U(A)$-modules
\begin{equation} \label{star0}
\U(B_0) Aug(\U( T)) \rightarrowtail \U( B_0) \epi \U(B) 
\end{equation}
where $Aug(\U( T)) = T \U(T) = \U(T) T$, hence $\U(B_0) Aug(\U( T))  = \U(B_0)  \U(T) T = \U(B_0) T$. For any  direct product we get another short exact sequence
of $\U(A)$-modules
$$
\prod \U(B_0) T  \rightarrowtail  \prod \U( B_0) \epi  \prod \U(B). 
$$
Consider the associated long exact sequence in homology
\begin{equation} \label{star}
\ldots \to H_{s}(A, \prod  \U( B_0) )  \to  H_s(A,  \prod \U( B) ) \to H_{s-1}(A,   \prod \U(  B_0) T )  \to \ldots
\end{equation}

\medskip
{\bf Claim.} {\it The map 
\begin{equation} \label{claim00}
H_{s}(A, \prod  \U( B_0) )  \to  H_s(A,  \prod \U( B) ) \hbox{ is an epimorphism}.\end{equation} }

\medskip

{\bf Proof of the Claim:} Assume first that $s\geq 2$. Recall that both $\U(B)$ and $\U(B_0)$ are free $\U(A)$-modules, hence 
\begin{equation} \label{zero1} H_s(A,  \U( B) ) = 0 = H_{s-1}(A,  \U( B_0) ) \hbox{ for }{s \geq 2}. \end{equation}
From the long exact sequence of homology associated to (\ref{star0})
\begin{equation}\label{star2}
\ldots \to H_s(A,  \U( B_0) ) \to H_s(A,  \U( B )) \to H_{s-1}(A,   \U(B_0) T )  \to H_{s-1}(A,  \U(B_0 )) \to \ldots
\end{equation}
and by (\ref{zero1}), we deduce that $$H_{s-1}(A,   \U(B_0) T ) =0 \hbox{ for } s \geq 2.$$
 Since $A$ is of type $FP_s$, we have
$$  H_{s-1}(A,   \prod  \U(B_0) T )  =  \prod H_{s-1}(A, \U(B_0) T ) = \prod 0 = 0$$
and the claim follows by the long exact sequence (\ref{star}).

For the case when $s=1$ we consider again the end of the sequence (\ref{star2})
$$
0 = H_1(A,  \U( B )) \to H_{0}(A,   \U(B_0) T  ) \to H_{0}(A,  \U(B_0) ) \to  H_{0}(A,  \U(B) )  \to 0,
$$
where $0 = H_1(A,  \U( B ))$ since $\U(B)$ is a free $\U(A)$-module. 
As before, as $A$ is finitely generated, the functor $H_0(A, - )$ commutes with direct products, so we get a short exact sequence
\begin{equation} \label{star12}
 H_{0}(A,  \prod \U(B_0) T  ) \rightarrowtail H_{0}(A, \prod \U (B_0) ) \epi  H_{0}(A, \prod \U( B) ).
\end{equation}
This together with the long exact sequence (\ref{star})
 completes the proof of the claim.

\medskip

The epimorphism (\ref{claim00}) together with the epimorphism $\nu : C_0 \to C$ induced by $\theta$, induce a surjective homomorphism $$\mu_2 : H_0(C_0, H_s(A, \prod \U( B_0)))  \to H_0(C, H_s(A, \prod \U( B))). $$  
Consider the commutative diagram 

$$\begin{tikzcd} H_{s+1} (C_0, \prod \U (C_0) ) \arrow[r,"\widetilde{d}_{s+1,0}^{s+1}"]  \arrow[d,"\mu_1"]& H_0(C_0, H_s(A, \prod \U( B_0)))\arrow[d,"\mu_2"]  \\
H_{s+1} (C, \prod \U(C))   \arrow[r,"d_{s+1,0}^{s+1}"] &  H_0(C, H_s(A, \prod \U(B)))  \end{tikzcd},
$$
where the horizontal maps are differential from Lyndon-Hoschild-Serre spectral sequences and the vertical maps are induced by the homomorphism $\nu$. 
Since $B_0$ is a Lie algebra of type $FP_{s+1}$ by Proposition  \ref{homology1} the differential  $\widetilde{d}_{s+1,0}^{s+1}$ is surjective. Since $\mu_2$ is surjective,  $\mu_2 \widetilde{d}_{s+1,0}^{s+1} = {d}_{s+1,0}^{s+1} \mu_1$ is surjective, hence ${d}_{s+1,0}^{s+1}$ is surjective. Then applying again Proposition \ref{homology1} we deduce that $B$ is a Lie algebra  
of type $FP_{s+1}$.
\end{proof}

\begin{proposition} \label{split} The homological $s$-$s+1$-$s+2$ Conjecture for Lie algebras  holds if one of the following  conditions is true :
\begin{itemize}
\item[a)] the second sequence splits;  
\item[b)]  we know that the conjecture holds whenever $L_2$ is a free Lie algebra.
\end{itemize}
\end{proposition}

\begin{proof} 

\noindent {\bf Case a): when the second sequence splits.} This means that $L_2 = B \leftthreetimes Q$ with $B$ an ideal of $L_2$. A Lie algebra version of the  proof of \cite[corollary~4.6]{Beno} implies that there is a monorphism $L_1 \to P = B \leftthreetimes L_1$ whose restriction to $A$ is the identity map. Then we get a commutative diagram of Lie algebra homomorphisms with exact rows

$$\begin{tikzcd}A \arrow[r, tail]\arrow[d,"1_d"] &L_1 \arrow[r,twoheadrightarrow]\arrow[d,"\theta"]& Q\arrow[d,"\nu"]\\
A \arrow[r, tail] &P\arrow[r,twoheadrightarrow]& L_2\\
\end{tikzcd}$$
with $\theta$ a monomorphism, so we can apply Theorem \ref{homologicalkuckuck} a).

\noindent {\bf Case b): when we know that the result holds true if $L_2$ is free.} Let $L_2$ be any Lie algebra of type $FP_{s+1}$ and $F$ a finitely generated free Lie algebra mapping onto $L_2$. We can apply a Lie algebra version of the proof of \cite[proposition~ 4.8]{Beno} to get a commutative diagram of Lie algebra homomorphisms with exact rows
$$\begin{tikzcd}A \arrow[r, tail]\arrow[d,"1_d"] &P' \arrow[r,twoheadrightarrow]\arrow[d,"\theta"]& F\arrow[d,"\nu"]\\
A \arrow[r, tail] &P\arrow[r,twoheadrightarrow]& L_2\\
\end{tikzcd},$$
where $P'$ is the fibre Lie sum of the short exact sequences of Lie algebras $A\mono L_1\epi Q$ and $B'\mono F\epi Q$. Since $\theta$ is an epimorphism we get the result 
using Theorem \ref{homologicalkuckuck} b).
\end{proof}

{\bf Remark.} In the above when  the second sequence splits it is enough to assume that $Q$ has type $FP_{s+1}$.

\medskip



\begin{theorem} 
If the homological $m-1$-$m$-$m+1$-Conjecture for Lie algebras holds  then the Homological Surjection Conjecture for Lie algebras holds too.
\end{theorem}

\begin{proof} The proof is similar to the proof of Theorem \ref{fin-pres}. 
Let  $m,k$ be natural numbers such that $ 2 \leq m \leq k$ and $$L \leq L_1 \oplus \ldots \oplus L_k$$ be a subdirect Lie sum 
 with $L \cap L_i \not= 0$, where each Lie algebra $L_i$ is $FP_m$ and for every $1 \leq i_1 < i_2 \ldots < i_m \leq k$ we have for the canonical projection $p_{i_1, \ldots, i_m} : L_1 \oplus \ldots \oplus L_k \to L_{i_1} \oplus \ldots \oplus L_{i_m}$ that $p_{i_1, \ldots, i_m}(L) = L_{i_1} \oplus \ldots \oplus L_{i_m}$. We have to prove that $L$ is a Lie algebra  of type $FP_m$.
The first paragraph of the proof of  Theorem \ref{fin-pres} implies that $\gamma_{k-1}(L_i) \subseteq L$. 

We prove the result by induction on $k \geq 2$, the case $k = 2$  is obvious so we assume $k \geq 3$. Let $Q=L_k/L\cap L_k,$ so 
 $Q$ is a finitely generated nilpotent Lie algebra, hence of finite dimension. Consider the short exact sequence of Lie algebras
\begin{equation} \label{ses01n} L \cap (L_1 \oplus \ldots \oplus L_{k-1}) \rightarrowtail  p_{1,2, \ldots , k-1} (L) \epi Q\end{equation}
 defined in the proof of Theorem \ref{fin-pres}. Then we see that $L$ is the fibre sum  of the short exact sequence (\ref{ses01n}) and
$L \cap L_k \mono L_k \epi Q.$ 

Note that if $m = k$ then $L = p_{1,2, \ldots, m}(L) = L_1 \oplus \ldots \oplus L_m$, so there is nothing to prove. Then we can assume that $m \leq k-1$.
 Since $p_{1,2, \ldots , k-1} (L)$ is a subdirect Lie  sum of $L_1 \oplus \ldots \oplus L_{k-1}$  that maps surjectively on the direct sum of any $m$ Lie algebras we deduce by induction on $k$ that  $p_{1,2, \ldots , k-1} (L)$ is $FP_m$. If we show that  $L \cap (L_1 \oplus \ldots \oplus L_{k-1})$ is a  Lie algebra of type $FP_{m-1}$ and assuming that the homological $m-1$-$m$-$m+1$-Conjecture for Lie algebras holds    we can  deduce that $L$ is $FP_m$ as required. By the proof of Theorem \ref{fin-pres} $L \cap (L_1 \oplus \ldots \oplus L_{k-1})$ is a finitely generated Lie algebra.
Furthermore since $p_{i_1, \ldots, i_m}(L) = L_{i_1} \oplus \ldots \oplus L_{i_m}$ for $i_m = k$, we have that
$$p_{i_1, \ldots, i_{m-1}}(L \cap (L_1 \oplus \ldots \oplus L_{k-1})) = L_{i_1} \oplus \ldots \oplus L_{i_{m-1}} \hbox{ for }1 \leq i_1 < \ldots < i_{m-1} \leq k-1.$$
Applying the induction again we obtain that $L \cap (L_1 \oplus \ldots \oplus L_{k-1})$ is a  Lie algebra of type $FP_{m-1}$.
\end{proof}

\end{document}